\documentclass[11pt]{article}
\usepackage[letterpaper,margin=1in]{geometry}
\usepackage[T1]{fontenc}
\usepackage{lmodern}
\usepackage{amsmath,amssymb,amsthm,mathtools}
\usepackage{booktabs,array,tabularx}
\usepackage[round,authoryear]{natbib}
\usepackage{microtype}
\usepackage{enumitem}
\usepackage{xcolor}

\usepackage{amsmath,amsfonts,bm}

\def\eqref#1{Equation~(\ref{#1})}

\def\1{\bm{1}}

\def\eps{{\epsilon}}

\DeclareMathAlphabet{\mathsfit}{\encodingdefault}{\sfdefault}{m}{sl}
\SetMathAlphabet{\mathsfit}{bold}{\encodingdefault}{\sfdefault}{bx}{n}

\newcommand{\E}{\mathbb{E}}

\newcommand{\R}{\mathbb{R}}

\usepackage{amsmath,amssymb,amsthm,mathtools}
\usepackage{booktabs,array,tabularx}
\usepackage{enumitem}

\usepackage[colorlinks,linkcolor=red,urlcolor=red,citecolor=blue,backref]{hyperref}

\setlist[enumerate]{leftmargin=*,itemsep=2pt,topsep=3pt}
\newtheorem{theorem}{Theorem}[section]
\newtheorem{proposition}[theorem]{Proposition}
\newtheorem{lemma}[theorem]{Lemma}
\newtheorem{corollary}[theorem]{Corollary}
\theoremstyle{definition}
\newtheorem{definition}[theorem]{Definition}
\newtheorem{remark}[theorem]{Remark}
\newcommand{\Prob}{\mathbb P}
\newcommand{\norm}[1]{\left\lVert #1\right\rVert}
\newcommand{\ip}[2]{\left\langle #1,#2\right\rangle}
\newcommand{\Lb}{\overline L}

\newcommand{\EE}{\mathbb E}

\usepackage{booktabs,array,tabularx,multirow,threeparttable}
\newcommand{\methodcell}[2]{%
  \begingroup
  \renewcommand{\arraystretch}{1}%
  \begin{tabular}[c]{@{}l@{}}
    #1\\[-1pt]
    {\footnotesize #2}
  \end{tabular}%
  \endgroup
}

\title{\bf Tight Lower Bounds for Stochastic Nonconvex-Strongly-Concave Minimax Optimization}

\author{
Siqi Zhang \thanks{School of Management and Engineering, Nanjing University, \texttt{siqi.zhang@nju.edu.cn}}
\and
Qilong Wu \thanks{School of Data Science, the Chinese University of Hong Kong (Shenzhen), \texttt{qilongwu@link.cuhk.edu.cn}}
\and
Junchi Yang \thanks{School of Data Science, the Chinese University of Hong Kong (Shenzhen), \texttt{yangjunchi@cuhk.edu.cn}}
}
\date{\today}

\begin{document}

\maketitle

\begin{abstract}
We study the stochastic first-order oracle complexity of finding $\epsilon$-stationary points of the primal function in smooth nonconvex-strongly-concave minimax optimization. For sufficiently small $\epsilon$, we establish lower bounds of $\Omega(\kappa L\Delta\sigma^2\epsilon^{-4})$ under the bounded-variance assumption and $\Omega(\kappa^{3/2}\bar L\Delta\sigma\epsilon^{-3})$ under the additional assumption of averaged smoothness. Here, $L$ and $\bar L$ denote the smoothness and averaged-smoothness constants, respectively, $\Delta$ is the initial primal gap, $\sigma^2$ bounds the oracle variance, and $\kappa=L/\mu$ or $\bar L/\mu$ in the respective settings, where $\mu$ is the strong-concavity parameter. Our bounded-variance lower bound improves the dependence on the condition number from $\kappa^{1/3}$ in previous lower bounds to $\kappa$, while our averaged-smoothness lower bound is the first of its kind. In both settings, the resulting lower bounds match existing upper bounds in their dependence on $\kappa$ and $\epsilon$. Our proofs are based on a unified quadratic lifting construction that transfers a hardness instance for stochastic nonconvex minimization to unconstrained minimax optimization while preserving the required variance and smoothness properties.
\end{abstract}

\section{Introduction}\label{sec:intro}
We consider the nonconvex-strongly-concave (NC-SC) minimax problem
\begin{equation}\label{eq:problem}
 \min_{x\in\R^{d_x}}\Phi(x),\qquad
 \Phi(x):=\max_{y\in\R^{d_y}} f(x,y),
\end{equation}
where $f$ has an $L$-Lipschitz gradient and $f(x,\cdot)$ is $\mu$-strongly concave. Problems of this form arise in a variety of machine learning applications~\citep{sinha2018certifying,liu2020stochastic,zhu2023provable}. Strong concavity ensures that the inner maximizer is unique, and the resulting primal function $\Phi$ is differentiable, although generally nonconvex. A central goal is to find a point $\hat x$ such that $\norm{\nabla\Phi(\hat x)}\le\eps$, using stochastic gradients of $f$. In the deterministic and finite-sum settings, accelerated proximal-point methods achieve the optimal complexity $\mathcal{O}(\sqrt{\kappa}\epsilon^{-2})$, where $\kappa\triangleq L/\mu$ is the condition number, already matching the corresponding lower bounds~\citep{lin2020near,zhang2021,li2021,han2024}.

In the stochastic setting, where $f(x,y)\triangleq \EE_\xi[F(x,y;\xi)]$, two common oracle assumptions lead to different complexity guarantees. Under an unbiased stochastic gradient oracle with uniformly bounded variance, existing algorithms achieve complexity $\mathcal{O}(\kappa\eps^{-4})$~\citep{sapd2024}. With the additional assumption of averaged smoothness of the stochastic gradients, the best known upper bounds improve to $\mathcal{O}(\eps^{-3})$~\citep{luo2020,sapd2024,spde2026}. However, the corresponding lower bounds remain incomplete. Under bounded variance, \citet{li2021} established an $\Omega(\kappa^{1/3}\eps^{-4})$ lower bound for constrained NC-SC problems, leaving a gap in the dependence on $\kappa$ relative to the $\mathcal{O}(\kappa\eps^{-4})$ upper bound. Under averaged smoothness, to the best of our knowledge, no lower bound has been established specifically for NC-SC minimax optimization.
We therefore ask:
\begin{quote}
\textit{Can we establish tight lower bounds for stochastic NC-SC minimax optimization under bounded variance and averaged smoothness?}
\end{quote}

\paragraph{Contributions}
In this work, we answer both questions affirmatively. Our contributions are three-fold. First, let $\epsilon>0$ denote the target accuracy and assume $\Phi(0)-\inf_x \Phi(x)\le \Delta$. For an unbiased bounded-variance (BV) oracle\footnote{Concurrent work by \citet{zhou2026tight} independently establishes an $\Omega(\kappa\epsilon^{-4})$ lower bound for the BV setting. Our work differs in considering fully unconstrained settings, and additionally establishes lower bounds under averaged smoothness.} with variance bounded by $\sigma^2$, we establish the lower bound, with $\kappa=L/\mu$, 
\begin{equation}\label{eq:intro-bv}
 \Omega\!\left(\frac{\kappa L\sigma^2\Delta}{\eps^4}\right).
\end{equation}
Second, for an averaged-smooth (AS) oracle with Lipschitz constant $\bar L$, we establish, with $\kappa=\bar L/\mu$, 
\begin{equation}\label{eq:intro-as}
 \Omega\!\left(\frac{\Lb\Delta}{\eps^2}
 \min\left\{\frac{\kappa\sigma^2}{\eps^2},
             \frac{\kappa^{3/2}\sigma}{\eps}\right\}\right).
\end{equation}

We refer to Theorems~\ref{thm:bv} and~\ref{thm:as} for the formal statements, and summarize the comparison with existing results in Table~\ref{tab:literature} and Figure~\ref{fig:LB_UB_comparison}. In the BV setting, our lower bound (\ref{eq:intro-bv}) improves the dependence on $\kappa$ over the previous lower bound of \citet{li2021} and matches the upper complexity bounds of \citet{sapd2024} and \citet{spde2026} in their dependence on both $\kappa$ and $\epsilon$. To the best of our knowledge, (\ref{eq:intro-as}) is the first lower bound specifically for the AS setting. Moreover, this lower bound matches the upper complexity bound of \citet{spde2026} in its dependence on $\kappa$ and $\epsilon$ when $\epsilon^{-1}\gtrsim\sqrt{\kappa}$, and that of \citet{sapd2022} when $\epsilon^{-1}\lesssim\sqrt{\kappa}$.

Finally, we note that the corresponding upper bounds are stated under slightly different problem classes and measures of initial optimality. Specifically, \citet{sapd2024} and \citet{spde2026} use a Moreau-envelope-based stationarity measurement. To facilitate a more direct comparison, we also establish lower bounds using a Moreau envelope stationarity, see Corollary~\ref{cor:envelope-main}.

\begin{figure}[ht]
    \centering
    \includegraphics[width=\linewidth]{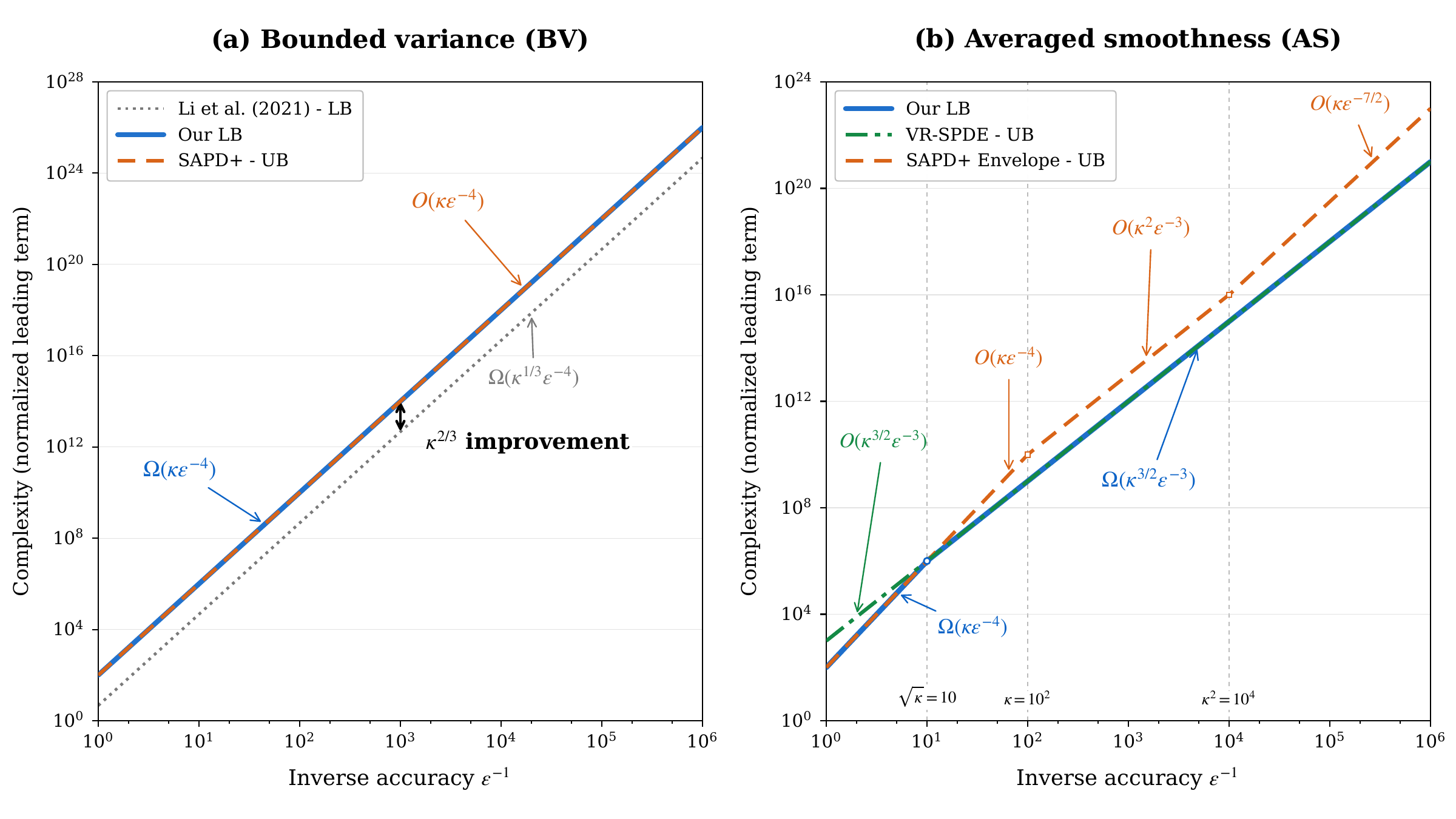}
    \caption{Schematic comparison of previous upper (UB) and lower bounds (LB) with our results under (a) bounded variance and
        (b) averaged smoothness.
        Both panels use log-log axes with $r=1/\varepsilon$,
        $\kappa=100$, and $\sigma=1$, and other constants, parameters and lower-order terms are suppressed.
        The SAPD+ curve~\citep{sapd2024} in (b) is the pointwise minimum of its non-VR and VR rate expressions.
        Overlapping curves indicate matching leading-order
        dependence, not identical oracle costs.}
    \label{fig:LB_UB_comparison}
\end{figure}

\begin{table}[t]
\centering
\caption{Summary and Comparison of Lower and Upper Bound Results}
\label{tab:literature}
\setlength{\tabcolsep}{3pt}
\renewcommand{\arraystretch}{1.75}
\renewcommand{\tabularxcolumn}[1]{m{#1}}
\begin{tabularx}{\linewidth}{
    @{}
    p{1.2cm}
    p{4.2cm}
    p{5.5cm}
    >{\raggedright\arraybackslash}X
    @{}
}
\toprule
\textbf{Setting} & \textbf{Method} & \textbf{Metric} & \textbf{Oracle Complexity} \\
\midrule
\multirow{5}{*}{BV}
& \methodcell{SGDA}{\citep{lin2020gradient}}
& Primal Stationarity 
& $\mathcal{O}(\kappa^3\eps^{-4})$ 
\\[4pt]

& \methodcell{SAPD+}{\citep{sapd2024}}
& Moreau Envelope Stationarity 
& $\mathcal{O}(\kappa\eps^{-4})$ 
\\[4pt]

& \methodcell{SPDE}{\citep{spde2026}}
& Moreau Envelope Stationarity
& $\mathcal{O}(\kappa\eps^{-4})$ 
\\[4pt]
\cline{2-4}

& \methodcell{\textbf{Lower Bound}}{\citep{li2021}}
& Gradient Mapping & $\Omega(\kappa^{1/3}\eps^{-4})$\\

& \methodcell{\textbf{Lower Bound (Ours)}}{Theorem~\ref{thm:bv}}
& Both Primal and Moreau 
& $\Omega(\kappa\eps^{-4})$\\
\midrule
\multirow{4}{*}{AS}
& \methodcell{SREDA}{\citep{luo2020}}
& Primal Stationarity 
& $\mathcal{O}(\kappa^3\eps^{-3})$ 
\\[4pt]

& \methodcell{SAPD+ (VR)}{\citep{sapd2024}}
& Moreau Envelope Stationarity
& $O\!\left(\kappa\eps^{-3}\max\{\kappa,\eps^{-1/2}\}\right)$
\\[4pt]

& \methodcell{VR-SPDE}{\citep{spde2026}}
& Moreau Envelope Stationarity
& $\mathcal{O}(\kappa^{3/2}\eps^{-3})$ 
\\[4pt]
\cline{2-4}

& \methodcell{\textbf{Lower Bound (Ours)}}{Theorem~\ref{thm:as}}
& Both Primal and Moreau 
& $\Omega(\min(\kappa^{3/2}\eps^{-3},\kappa\eps^{-4}))$ \\
\bottomrule
\end{tabularx}

\vspace{3pt}
\begin{minipage}{\linewidth}
\footnotesize

\textbf{Note:} BV denotes the bounded-variance setting, and AS denotes the averaged-smoothness setting. We use $\kappa=L/\mu$ in the BV setting and $\kappa=\Lb/\mu$ in the AS setting. See Theorems~\ref{thm:bv} and~\ref{thm:as} for the precise assumptions and problem settings. AS is a special case of BV, and hence upper bounds established for BV also apply to AS.
\end{minipage}
\end{table}

The hard instance is obtained through a quadratic lifting of hard instances for stochastic nonconvex minimization~\citep{carmon2020,arjevani2023}. Starting from the hard instance $F_T$ by \citet{carmon2020}, we construct a minimax objective of the form
\[
f_T(x,y)=\alpha F_T(\beta y)-\frac{\nu}{2}\|y-\gamma x\|^2,
\]
where the parameters $\alpha, \beta, \nu$ and $\gamma$ are chosen so that $f_T(x,\cdot)$ is strongly-concave. The quadratic coupling is designed so that stationarity of the primal function $\Phi_T(x)=\max_y f_T(x,y)$ implies approximate stationarity of $F_T$ at the explicitly rescaled point $\beta\gamma x$. In particular, an $\eps$-stationary point $\hat x$ of $\Phi_T$ yields an appropriately scaled approximate stationary point $\beta\gamma\hat x$ of $F_T$. Since this point is obtained directly from $\hat x$, the reduction does not require computing the inner maximizer and therefore introduces no additional oracle cost. The construction is compatible with stochastic oracle: we keep the quadratic coupling deterministic and introduce randomness only through the stochastic oracle associated with $F_T$, preventing the oracle variance from growing with $\|x\|$ or $\|y\|$ over the unconstrained domain. The main technical challenge is therefore to establish this stationarity-preserving reduction and tune the scaling parameters so that the lifted instance simultaneously satisfies the required smoothness, strong-concavity and variance conditions.

\section{Related Literature}\label{sec:related}

\paragraph{Lower Bounds for Minimax Optimization.}
First-order oracle lower bounds for convex-concave and strongly-convex-strongly-concave minimax problems were established by \citet{ouyang2021} and \citet{zhang2022saddle}, while \citet{xie2020} studied the finite-sum setting. For nonconvex-strongly-concave (NC-SC) problems, \citet{zhang2021} derived lower bounds for both the deterministic and averaged-smooth finite-sum settings. \citet{han2024} further studied finite-sum NC-SC problems under an oracle model that allows access to both component gradients and proximal operators. Closest to our setting, \citet{li2021} established deterministic and stochastic lower bounds for NC-SC problems; under bounded variance, their stochastic lower bound contains the term $\Omega(L\Delta\kappa^{1/3}\sigma^2\eps^{-4})$. For a broader overview of oracle-complexity results in minimization and minimax optimization, we refer to \citet{zhanghu2025overclaims}.

\paragraph{Nonconvex Minimax Optimization Algorithms.}
In the deterministic NC-SC setting, accelerated proximal-point methods achieve near-optimal dependence on the condition number~\citep{lin2020near,zhang2021}. In the stochastic setting, two-timescale SGDA attains $\mathcal{O}(\kappa^3\eps^{-4})$ stochastic gradient complexity under bounded variance~\citep{lin2020gradient,lin2025}, while SREDA achieves $\mathcal{O}(\kappa^3\eps^{-3})$ under averaged smoothness~\citep{luo2020}. Momentum-based and adaptive variants have also been developed~\citep{huang2022,huang2023}. SAPD+ combines a proximal-point framework with accelerated saddle-point solvers and also admits a variance-reduced variant~\citep{sapd2024}. More recently, \citet{spde2026} proposed the single-loop methods SPDE and VR-SPDE, with complexities $\mathcal{O}(\kappa\eps^{-4})$ under bounded variance and $\mathcal{O}(\kappa^{3/2}\eps^{-3})$ under averaged smoothness, respectively. These rates match the $(\kappa,\eps)$ dependence of the stochastic terms in our lower bounds. Table~\ref{tab:literature} summarizes the rates of these methods. Beyond the NC-SC setting, \citet{li2026smoothing} study smoothing and dual perturbation for nonconvex-concave problems, while \citet{zhangxu2026} develop single-loop methods that attain rates previously achieved by multi-loop schemes. Recent work on NC-SC problems considers generalized smoothness~\citep{gao2026} and regularized Newton methods~\citep{ma2026}.

\paragraph{Recent Lower Bounds beyond Strong Concavity.}
Recent work has extended lower-bound analyses to minimax problems in which strong concavity is relaxed to concavity or the Polyak-{\L}ojasiewicz (PL) condition. For NC-C problems with a compact dual domain, \citet{wu2026ncc} establish deterministic and stochastic lower bounds of $\Omega(\eps^{-3})$ and $\Omega(\eps^{-6})$, respectively, for zero-respecting algorithms under Moreau-envelope stationarity. Deterministic lower bounds for projected zero-respecting methods are also developed by \citet{zhangxu2026}. \citet{pan2026ncc} establish the $\Omega(\eps^{-3})$ deterministic lower bound for arbitrary first-order algorithms and provide a matching upper bound. For NC-PL problems, \citet{panli2026pl} establish an $\Omega(L\Delta\kappa\eps^{-2})$ lower bound for arbitrary deterministic first-order methods, matching the known upper bound in the corresponding parameter regime. In particular, their result demonstrates that relaxing strong concavity to the PL condition can lead to a stronger dependence on the condition number $\kappa$ in the deterministic setting. In contrast, our work focuses on stochastic NC-SC problems and characterizes how the bounded-variance and averaged-smoothness oracle models determine the dependence on $\kappa$.

\section{Problem Setup and Oracle Complexity}\label{sec:setup}
We use Euclidean norms, denote $z=(x,y)$, and $\operatorname{supp}(z)\triangleq\{i:z_i\neq 0\}$ denote the support of $z\in\mathbb{R}^d$. For a mapping $T$ between Euclidean spaces, we denote
\[
\operatorname{Lip}(T)
\triangleq
\inf\left\{
C\ge0:
\|T(u)-T(v)\|\le C\|u-v\|,
\ \forall\,u,v
\right\}
\]
as its global Lipschitz constant.
Here we study the lower bound under the oracle complexity framework~\citep{nemirovskij1983problem}, which comes with the following components: a function class, a stochastic oracle,
an algorithm class, and a convergence criterion. We will discuss them in details in this section.

\paragraph{Function class.}
We consider nonconvex-strongly-concave (NC-SC) objectives. For $M,\mu,\Delta>0$, let $\mathcal F^{d_x,d_y}(M,\mu,\Delta)$ denote the set of continuously differentiable functions $f:\R^{d_x}\times\R^{d_y}\to\R$ such that, with $z=(x,y)$, $\Phi(x)\triangleq\max_y f(x,y)$, $\Phi_*\triangleq\inf_x\Phi(x)>-\infty$,
\begin{align}
 &\norm{\nabla f(z)-\nabla f(z')}
 \le M\norm{z-z'}
 && \forall z,z'\in\R^{d_x+d_y},
 \label{eq:joint-smooth}\\
 &f(x,y')
 \le f(x,y)+\ip{\nabla_y f(x,y)}{y'-y}
 -\frac{\mu}{2}\norm{y'-y}^2
 && \forall x\in\R^{d_x},\; y,y'\in\R^{d_y},
 \label{eq:strong-concavity}\\
 &\Phi(0)-\Phi_*
 \le \Delta.
 \label{eq:primal-gap}
\end{align}
Both the primal and dual variables are unconstrained, and no convexity in $x$ is imposed. Strong concavity ensures that the inner maximizer $y^*(x)$ exists and is unique for every $x$. Moreover, $\Phi$ is differentiable with
$\nabla\Phi(x)=\nabla_x f(x,y^*(x))$, is $M$-weakly convex, and has an $M(1+M/\mu)$-Lipschitz gradient. These properties hold without any compactness assumption; see Appendix~\ref{app:lift}.

\paragraph{Oracle class.}
We consider \emph{stochastic first-order oracle (SFO)}~\citep{arjevani2023}. An instance is a pair $(f,G)$, including the law of the random seed $\xi$.
The oracle is a measurable random vector field and returns a joint gradient estimate
$G(z;\xi)=(G_x(z;\xi),G_y(z;\xi))$ such that, for every $z$,
\begin{equation}\label{eq:oracle-bv}
 \E_\xi G(z;\xi)=\nabla f(z),\qquad
 \E_\xi\norm{G(z;\xi)-\nabla f(z)}^2\le\sigma^2.
\end{equation}
The \emph{bounded-variance} (BV) class $\mathcal I_{\rm BV}(L,\mu,\Delta,\sigma)$
requires $f\in\mathcal F^{d_x,d_y}(L,\mu,\Delta)$ and~\eqref{eq:oracle-bv}.
The \emph{averaged-smooth} (AS) class $\mathcal I_{\rm AS}(\Lb,\mu,\Delta,\sigma)$
requires $f\in\mathcal F^{d_x,d_y}(\Lb,\mu,\Delta)$,~\eqref{eq:oracle-bv}, and
\begin{equation}\label{eq:oracle-as}
 \E_\xi\norm{G(z;\xi)-G(z';\xi)}^2\le\Lb^2\norm{z-z'}^2
 \qquad\text{for all }z,z'.
\end{equation}
Here, the same random seed $\xi$ is used to evaluate $G$ at $z$ and $z'$. By unbiasedness and Jensen's inequality,
\(
 \norm{\nabla f(z)-\nabla f(z')}^2
 \le
 \E_\xi\norm{G(z;\xi)-G(z';\xi)}^2\), 
so~\eqref{eq:oracle-as} implies that $\nabla f$ is $\Lb$-Lipschitz. Accordingly, we define the condition number as $\kappa=L/\mu$ in the BV setting and $\kappa=\Lb/\mu$ in the AS setting.

\subsection{Algorithm Class and Lower Bound}
\paragraph{Oracle access.}
At round $t$, an algorithm selects $1\le K_t\le K$ points
$z^{t,1},\ldots,z^{t,K_t}$ using its past responses and internal randomness.
A fresh independent seed $\xi_t$ is then drawn, and the oracle returns
$G(z^{t,1};\xi_t),\ldots,G(z^{t,K_t};\xi_t)$.
The current points are chosen before any current response is observed;
the seed is hidden and cannot be queried again later.
We take $K=1$ in BV and any fixed finite $K$ in AS; $K=2$ permits paired
same-sample differences. 

\paragraph{Algorithm class.}
We consider \emph{zero-respecting} stochastic first-order algorithms.
Let $g^{t,k}=G(z^{t,k};\xi_t)$ and define the coordinates revealed before round $t$ by
\begin{equation}\label{eq:revealed-support}
 \mathcal S_t=\bigcup_{s<t}\ \bigcup_{j=1}^{K_s}\operatorname{supp}(g^{s,j}).
\end{equation}
An algorithm is zero-respecting if, for every admissible instance, its queries and
its output after $R$ rounds satisfy
\begin{equation}\label{eq:zero-respecting}
 \operatorname{supp}(z^{t,k})\subseteq\mathcal S_t,\qquad
 \operatorname{supp}((\hat x,0))\subseteq\mathcal S_{R+1}.
\end{equation}
Here $(\hat x,0)\in\R^{d_x+d_y}$ denotes the embedding of the primal output into the joint primal-dual space. The first condition requires every query to use only coordinates revealed by previous oracle gradients, while the second requires every nonzero coordinate of the final output $\hat x$ to have been revealed by the end of round $R$. In particular, let $\mathcal S_1=\emptyset$, so first-round queries are at the origin. It covers a lot of algorithms including linear-span algorithms and coordinate-wise updates that preserve zero coordinates.

\paragraph{Complexity measurement.}
Our goal is to find the stationary point $\hat{x}$ in terms of the expectation, i.e.,
$$\E\norm{\nabla\Phi(\hat x)}^2\le\eps^2;$$
we also give lower bounds for $\E\norm{\nabla\Phi(\hat x)}\le\eps$ (refer to Proposition~\ref{prop:canonical-hardness}).
After $R$ rounds, the number of oracle calls is $N=\sum_{t=1}^R K_t$.
Each returned joint gradient vector counts once, including all minibatch and
refresh samples. Let $\mathcal A_{{\rm zr},N}^{(K)}$ be the algorithms above using
at most $N$ calls on every run. For either instance class $\mathcal I$, define
\begin{equation}\label{eq:complexity}
 \mathcal C_{\eps,{\rm zr}}^{(K)}(\mathcal I)
 =\inf\!\left\{N\in\mathbb N_0:
 \exists A\in\mathcal A_{{\rm zr},N}^{(K)},\
 \sup_{I\in\mathcal I}\E_I\norm{\nabla\Phi_I(\hat x_A)}^2\le\eps^2\right\}.
\end{equation}
The classes range over all finite dimensions; algorithms may know the dimensions
and class parameters. Thus a lower bound rules out a uniform guarantee for every
algorithm in this class, not just a particular update rule.
Our proof bounds rounds first and uses $R\le N$ to bound calls, without an extra
factor $K$. The expected-norm version is defined by replacing the squared norm
and $\eps^2$ in~\eqref{eq:complexity} by the norm and $\eps$.

\subsection{Hard Instance Construction}
\label{subsec:hard_instance_cons}
Let $H:\R^d\to\R$ have an $\ell$-Lipschitz gradient, and let $g_H$ be an unbiased
estimator of $\nabla H$. For positive $h,q,\gamma$, set
\begin{equation}\label{eq:lift-parameters}
 \alpha=q^2/h,\qquad \beta=h/q,\qquad \nu=\mu+\ell h,
\end{equation}
and define the minimax objective
\begin{equation}\label{eq:lift}
 f_H(x,y)=\alpha H(\beta y)-\frac{\nu}{2}\norm{y-\gamma x}^2.
\end{equation}
We evaluate the quadratic gradient exactly and randomize only the estimate of
$\nabla H$:
\begin{equation}\label{eq:lift-oracle}
 G_x=\nu\gamma(y-\gamma x),\qquad
 G_y=qg_H(\beta y;\xi)-\nu(y-\gamma x).
\end{equation}
If $x_i=y_i=0$, then $[G_x]_i=0$ and $[G_y]_i=q[g_H(\beta y;\xi)]_i$. Thus, the quadratic coupling cannot
reveal a new coordinate on its own; new coordinates enter the support
only through the base oracle $g_H$. The following proposition is the main reduction used in proofs.

\begin{proposition}[Lifted Function]\label{prop:lift}
Suppose $H$ is twice continuously differentiable and bounded below, with
$H(0)-\inf H\le D$ and $\norm{\nabla H(0)}\le g$.
Then $f_H(x,\cdot)$ is $\mu$-strongly concave and
\begin{align}
 \operatorname{Lip}(\nabla f_H)&\le\nu(1+\gamma^2)+\ell h,\label{eq:lift-smooth}\\
 \Phi_H(0)-\inf\Phi_H&\le\alpha D+\frac{q^2g^2}{2\mu},\label{eq:lift-gap}\\
 \norm{\nabla H(\beta\gamma x)}&\le\frac{1+\ell h/\nu}{\gamma q}\norm{\nabla\Phi_H(x)}
 \qquad\forall x.\label{eq:stationarity-transfer}
\end{align}
If $g_H$ has variance at most $v^2$, the lifted oracle has variance at most $q^2v^2$.
If $g_H$ has averaged-smoothness constant $S$, the lifted constant is at most
$\nu(1+\gamma^2)+hS$.
\end{proposition}

\begin{remark}
    The above result is important and characterizes several important properties of the lifted function $f_H$. We want to highlight \eqref{eq:stationarity-transfer} above, with this lower bound result, we transfer the lower bound study on the primal function $\nabla\Phi_H$ to that of $\nabla H$, which is well-studied in nonconvex minimization lower bound literature~\citep{carmon2020,arjevani2023}.
    Moreover, note that
    \begin{equation}\label{eq:noise-identity}
     G(x,y;\xi)-\nabla f_H(x,y)
     =\bigl(0,q[g_H(\beta y;\xi)-\nabla H(\beta y)]\bigr).
    \end{equation}
    Thus the unbounded quadratic gradient contributes no noise, which allows a
    uniform variance bound on the whole space.
\end{remark}

\paragraph{Choice of $H$ and $g_H$.}
Regarding the hard instance function $H$ and oracle $g_H$ for lower bound derivations, we take $H:\mathbb{R}^T\rightarrow \mathbb{R}$ to be
the hard instance function introduced by \citet{carmon2020} (i.e., the function $F_T$ therein), which satisfies the zero-chain property,
and used in the stochastic nonconvex lower bounds of \citet{arjevani2023}.
For an integer $T\ge2$, it is defined by
\begin{equation}\label{eq:base-chain-main}
\begin{aligned}
H(u)
={}-\Psi(1)\Phi_0(u_1)
+\sum_{i=2}^{T}
\bigl[
\Psi(-u_{i-1})\Phi_0(-u_i)
-\Psi(u_{i-1})\Phi_0(u_i)
\bigr],
\end{aligned}
\end{equation}
where
\[
\Psi(t)=
\begin{cases}
0, & t\le \frac12,\\
\exp \bigl(1-(2t-1)^{-2}\bigr), & t>\frac12,
\end{cases}
\qquad
\Phi_0(t)=\sqrt e\int_{-\infty}^{t}e^{-s^2/2}\,ds.
\]
Each summand couples two consecutive coordinates.
Starting from the origin, a zero-respecting method using
exact gradients can reveal at most one new coordinate
per query.
Moreover, $u_T=0$ implies $\|\nabla H(u)\|>1$, so the
method must reach the last coordinate before finding
a point with a small gradient.
The gradient Lipschitz constant is independent of $T$,
and the initial gap is $\mathcal{O}(T)$.

For the estimator $g_H$, we use the stochastic gradient constructions of
\citet{arjevani2023}. Additional properties of the resulting stochastic hard
instances are collected in Appendix~\ref{app:base}. In both the BV and AS
settings, the lifted constructions preserve the probabilistic zero-chain
structure of the underlying hard instance; we describe this property next. 

\subsection{Probabilistic Zero-Chain Mechanism}\label{subsec:zero-chain}
For $u\in\R^T$ and $a\ge0$, define
\[
 \operatorname{prog}_a(u)
 :=\max\bigl(\{i\in\{1,\ldots,T\}:|u_i|>a\}\cup\{0\}\bigr).
\]
In particular, $\operatorname{prog}_0(u)$ is the largest index in
$\operatorname{supp}(u)$, with value zero when $u=0$.

\begin{definition}[First-Order Zero-Chain]\label{def:zero-chain}
A differentiable function $H:\R^T\to\R$ is a first-order zero-chain if
\[
 \operatorname{prog}_0(\nabla H(u))
 \le \operatorname{prog}_0(u)+1
 \qquad\text{for every }u\in\R^T.
\]
\end{definition}
For a zero-respecting method starting from the origin, an exact-gradient
query can therefore reveal at most one new coordinate.
The stochastic counterpart controls how often a new coordinate can be revealed.
We use the following definition from \citet{arjevani2023}.

\begin{definition}[Probability-$p$ Zero-Chain]\label{def:pzc}
Let $p\in(0,1]$, and let $g_H(u;\xi)$ be an unbiased gradient estimator of $H$.
We call $g_H$ a probability-$p$ zero-chain (PZC) if
\begin{equation}\label{eq:pzc-definition}
\begin{aligned}
 \Prob_\xi\!\left(\exists u\in\R^T:
   \operatorname{prog}_0(g_H(u;\xi))
   =\operatorname{prog}_{1/4}(u)+1\right)&\le p,\\
 \Prob_\xi\!\left(\exists u\in\R^T:
   \operatorname{prog}_0(g_H(u;\xi))
   >\operatorname{prog}_{1/4}(u)+1\right)&=0.
\end{aligned}
\end{equation}
\end{definition}
A single seed is used for all points inside each event.
Thus the definition controls an entire same-seed batch, not just each
query separately.

\begin{lemma}[Progress of the Lifted Oracle]\label{lem:progress}
Suppose $g_H$ is a probability-$p$ zero-chain in the sense of
Definition~\ref{def:pzc}, and let $G$ be the lifted oracle
in~\eqref{eq:lift-oracle}, with $d_x=d_y=T$.
For any zero-respecting algorithm following the oracle protocol of
Section~\ref{sec:setup} and returning $\hat x$ after at most $R$ rounds,
\[
 \Prob\!\left(\operatorname{supp}(\hat x)
      \subseteq\{1,\ldots,T-1\}\right)\ge 1-\frac{pR}{T}.
\]
In particular, $R\le T/(4p)$ implies $\Prob(\hat x_T=0)\ge3/4$.
\end{lemma}
This adapts the PZC progress argument of
\citet[Lemma~1]{arjevani2023} to the quadratic lift.
Appendix~\ref{app:progress} detailed the discussion and characterizes the PZC property of the two base oracles,
proves Lemma~\ref{lem:progress}, and combines it with
\eqref{eq:stationarity-transfer} to obtain the gradient lower bound in
Proposition~\ref{prop:canonical-hardness}.

\section{The Bounded Variance Case Lower Bound}\label{sec:bv}
In this section, we derive the lower complexity bound of stochastic NC-SC minimax optimization problems, under the generic bounded variance setting.

\begin{theorem}[Bounded Variance]\label{thm:bv}
There are universal constants $c,c_0>0$ such that, for
$L,\mu,\Delta>0$, $\sigma\ge0$, $\kappa=L/\mu\ge8$, and
$0<\eps^2\le c_0L\Delta$,
\begin{equation}\label{eq:bv-lower}
 \mathcal C_{\eps,{\rm zr}}^{(1)}
 \bigl(\mathcal I_{\rm BV}(L,\mu,\Delta,\sigma)\bigr)
 \ge c\frac{L\Delta}{\eps^2}
       \max\left\{1,\frac{\kappa\sigma^2}{\eps^2}\right\}.
\end{equation}
The same order holds for the expected-norm target.
The hard instances have actual smoothness-to-strong-concavity ratio within
universal constant factors of $\kappa$.
\end{theorem}

\paragraph{Proof Sketch.}
We use the hard instance function and oracle as discussed in Section~\ref{subsec:hard_instance_cons}, which are motivated by \citet{carmon2020,arjevani2023}. Its gradient stays bounded away from zero
until its last coordinate is reached. The oracle reveals the next coordinate
with probability at most $p$, with variance $\mathcal{O}(p^{-1})$ independent of $T$.
In~\eqref{eq:lift-parameters}, we choose
\begin{equation}\label{eq:bv-scaling}
 \gamma^2\asymp\kappa,\qquad q^2\asymp\eps^2/\kappa,\qquad h\asymp\mu,
 \qquad p\asymp\min\{1,\eps^2/(\kappa\sigma^2)\},
\end{equation}
where $p=1$ when $\sigma=0$.
Proposition~\ref{prop:lift} guarantees that the stationarity transfer fixes $\gamma q$ at a constant multiple of $\eps$. A zero-respecting algorithm needs order $T/p$ rounds to reveal the chain.
The dependence on $\kappa$ comes from the smaller reveal probability.
Appendices~\ref{app:base} to \ref{app:bv-calibration} present the detailed oracle definition, support
argument, and all parameter checks. 

\paragraph{Comparison with Existing Results.}
The previous lower bound for the BV setting is
$\Omega(\kappa^{1/3}\eps^{-4})$~\citep{li2021}. Our
Theorem~\ref{thm:bv} strengthens the dependence on the condition number from
$\kappa^{1/3}$ to $\kappa$, while allowing both the primal and dual variables
to remain unconstrained. On the upper-bound side, SAPD+ and SPDE achieve a
leading stochastic complexity of $\mathcal{O}(\kappa\eps^{-4})$ for
Moreau-envelope stationarity~\citep{sapd2024,spde2026}. In
Section~\ref{subsec:moreau}, we extend our lower bound to the same stationarity
criterion, while Section~\ref{subsec:upper-comparison} discusses the remaining
differences in assumptions.

\section{The Averaged-Smooth Lower Bound}\label{sec:as}
Building on the BV analysis, we now impose the additional averaged-smoothness assumption on the stochastic oracle and derive a sharper lower bound.

\begin{theorem}[Averaged Smoothness]\label{thm:as}
There are universal constants $c,c_0,c_1>0$ such that the following holds for
every fixed finite $K\ge1$. 
Let $\Lb,\mu,\Delta>0$, $\sigma\ge0$, $\kappa=\Lb/\mu\ge8$, and suppose
\begin{equation}\label{eq:as-accuracy}
 0<\eps^2\le c_0\Lb\Delta,\qquad
 \eps\sigma\le c_1\Lb\Delta\sqrt\kappa.
\end{equation}
Then
\begin{equation}\label{eq:as-lower}
 \mathcal C_{\eps,{\rm zr}}^{(K)}
 \bigl(\mathcal I_{\rm AS}(\Lb,\mu,\Delta,\sigma)\bigr)
 \ge c\frac{\Lb\Delta}{\eps^2}
 \max\left\{1,\min\left\{\frac{\kappa\sigma^2}{\eps^2},
                         \frac{\kappa^{3/2}\sigma}{\eps}\right\}\right\}.
\end{equation}
The same order holds for the expected-norm target.
\end{theorem}

\subsection{Proof Sketch}
The stochastic oracle used in the BV analysis does not, in general, satisfy the averaged-smoothness requirement. We therefore adopt the smooth-gating construction of \citet{arjevani2023}, which preserves variance of order $\mathcal{O}(p^{-1})$ while achieving an averaged-smoothness constant of order $\mathcal{O}(p^{-1/2})$. As a consequence, Proposition~\ref{prop:lift} imposes the additional constraint
\(
h\lesssim \min\{\mu,\Lb\sqrt{p}\}\).
With the same choices $\gamma^2\asymp\kappa$ and $q^2\asymp\eps^2/\kappa$ as in the BV case, the resulting complexity scales as
\begin{equation}\label{eq:as-mechanism}
 \frac{T}{p}
 \asymp
 \frac{\Delta h}{q^2p}
 \asymp
 \frac{\Lb\Delta}{\eps^2}
 \min\left\{p^{-1},\kappa p^{-1/2}\right\}.
\end{equation}
Choosing
$p\asymp\min\{1,\eps^2/(\kappa\sigma^2)\}$
then yields~\eqref{eq:as-lower}. The conditions in~\eqref{eq:as-accuracy} ensure that the chain length is admissible and that the initial primal gap remains within $\Delta$. Full calibration details are given in Appendix~\ref{app:as-calibration}.

Returning to~\eqref{eq:as-lower}, the lower bound can be expressed across three regimes:
\begin{equation}\label{eq:as-regimes}
 \Lb\Delta\eps^{-2}\times
 \begin{cases}
  1, &\sigma\eps^{-1}\le\kappa^{-1/2},\\
  \kappa \sigma^2\eps^{-2}, &\kappa^{-1/2}<\sigma\eps^{-1}<\sqrt\kappa,\\
  \kappa^{3/2}\sigma\eps^{-1}, &\sigma\eps^{-1}\ge\sqrt\kappa.
 \end{cases}
\end{equation}
In particular, in the high-accuracy regime, where $\eps$ is sufficiently small, the lower bound reduces to
$\Omega(\Lb\Delta\kappa^{3/2}\sigma\eps^{-3})$.
Moreover, when $L=\Lb$, the AS lower bound is always no larger than the corresponding BV lower bound.

\subsection{Lower Bounds for Moreau Envelope Stationarity}
\label{subsec:moreau}

Theorems~\ref{thm:bv} and~\ref{thm:as} measure stationarity using the
primal gradient. But recent NC-SC literature of SAPD+ and SPDE
\citep{sapd2024,spde2026} used Moreau envelope stationarity as the measurement.
Our primal stationarity lower bound results in Theorem~\ref{thm:as} does not provide a lower bound for this criterion, so a separate transfer is needed. 
Set $M=L$ in the BV setting and $M=\Lb$ in the AS setting, and define
\begin{equation}\label{eq:envelope}
 P_f(x):=\Phi_{1/(2M)}(x)
 =\min_{u\in\R^{d_x}}\{\Phi(u)+M\norm{u-x}^2\}.
\end{equation}
Lemma~\ref{lem:primal-envelope-properties} ensures that $P_f$ is
well-defined and continuously differentiable.

On the hard instances used in our proofs, the two stationarity
measures are equivalent up to universal constants.
Specifically, Lemma~\ref{lem:envelope-comparison} shows that
\begin{equation}\label{eq:envelope-key}
 \frac{123}{128}\norm{\nabla P_{f_H}(x)}
 \le \norm{\nabla\Phi_H(x)}
 \le \frac{133}{128}\norm{\nabla P_{f_H}(x)}
 \qquad\text{for every }x.
\end{equation}
Thus here a small Moreau-envelope gradient would also
imply a small primal gradient, with a constant factor independent of
$\kappa$. Combining this comparison with the primal-gradient lower
bounds in Proposition~\ref{prop:canonical-hardness}, we have the following result.

\begin{corollary}[Moreau Envelope Lower Bounds]\label{cor:envelope-main}
The lower bounds in Theorems~\ref{thm:bv} and~\ref{thm:as} also hold,
up to universal constant factors, when the required output satisfies either
\[
 \E\norm{\nabla P_f(\hat x)}^2\le\eps^2
 \qquad\text{or}\qquad
 \E\norm{\nabla P_f(\hat x)}\le\eps.
\]
Here the parameters, oracle protocols, and algorithm class are unchanged.
\end{corollary}

Appendix~\ref{app:moreau} proves the gradient comparison and the
corollary. Both gradients are evaluated at the algorithm's output,
so no additional oracle call or proximal computation is needed.
The resulting bounds preserve the dependence on $\kappa$ and $\eps$
and can be compared directly with upper bounds stated for
Moreau-envelope stationarity.

\subsection{Comparison with Existing Upper Bounds}
\label{subsec:upper-comparison}

For primal stationarity, SREDA established $\mathcal{O}(\kappa^3\eps^{-3})$ SFO complexity under averaged smoothness~\citep{luo2020}, while it further assumes a compact dual domain and sample-wise concavity.
Its expected-norm criterion is also covered by our lower bounds (refer to Proposition~\ref{prop:canonical-hardness}). For Moreau envelope stationarity, SAPD+~\citep[the revised version]{sapd2024} gives
\begin{equation}\label{eq:sapd-corrected}
 \mathcal{O}\!\left(\frac{L\mathcal G_0}{\eps^2}
       \max\{\kappa,\sqrt{\sigma/\eps}\}
       (1+\kappa\sigma/\eps)\right)
\end{equation}
for its variance-reduced variant \citep{sapd2024}, where $\mathcal G_0$ is the initial primal-dual gap.
Consequently, its bound is not uniformly $\mathcal{O}(\kappa^2\eps^{-3})$ as accuracy $\eps\rightarrow 0$.

More recently, \citet{spde2026} proposed VR-SPDE and obtained
$\mathcal{O}(\Lb \kappa^{3/2}\sigma\eps^{-3})$ for the AS case.
By Corollary~\ref{cor:envelope-main}, we can find that the result of \citet{spde2026} tightly matches our lower bound in the high-accuracy regime. Also as shown in Table~\ref{tab:literature} and Figure~\ref{fig:LB_UB_comparison}, in the low-accuracy regime, the $\mathcal{O}(\kappa\eps^{-4})$ upper bound by SAPD+~\citep{sapd2024} also tightly matches our lower bound, so our lower bound reveals the optimal SFO complexity in terms of $\kappa$ and $\eps$ for the AS case.

\section{Conclusion and Future Work}\label{sec:conclusion}

We establish stochastic first-order lower bounds for fully unconstrained NC-SC minimax optimization over zero-respecting algorithms. For sufficiently small $\eps$, we prove lower bounds of $\Omega(L\Delta\kappa\sigma^2\eps^{-4})$ under bounded variance and $\Omega(\Lb\Delta\kappa^{3/2}\sigma\eps^{-3})$ under averaged smoothness. Both results follow from a unified construction that couples a nonconvex zero-chain with a deterministic quadratic lifting, and both extend to Moreau envelope stationarity without worsening the dependence on $\kappa$. These lower bounds match recent stochastic upper bounds in terms of their dependence on $\kappa$ and $\eps$. 

Several questions remain open. First, it would be useful to extend the lower bounds beyond zero-respecting algorithms to more general, possibly randomized, first-order methods. Another direction is to better align the domain and initialization assumptions in upper and lower bounds. Finally, our zero-chain hard instances require a dimension that grows with the target accuracy; whether comparable lower bounds hold in fixed or low dimensions, in the spirit of \citet{chewi2023complexity}, remains open.

\section*{Lean Formalization}

A Lean formalization that provides machine-checked verification of key components of the analysis is available at: \href{https://github.com/Wu-Qilong/Tight-Lower-Bounds-for-Stochastic-Nonconvex-Strongly-Concave-Minimax-Optimization}{https://github.com/Wu-Qilong/Tight-Lower-Bounds-for-Stochastic-Nonconvex-Strongly-Concave-Minimax-Optimization}

\section*{AI Use Statement}

In this work, we used generative AI tools, primarily ChatGPT, to assist with exploring candidate hard instance constructions and proof strategies, refining mathematical arguments, surveying relevant literature, and drafting and polishing part of the manuscript. The authors formulated the research questions and problem settings, evaluated and selected among candidate approaches, and independently verified all mathematical claims and proofs included in the final paper. The authors have reviewed the work and take full responsibility for the final content of this paper.

\bibliographystyle{plainnat}
\bibliography{arxiv_ref}
\clearpage

\appendix
\section{Preliminaries and proof of the quadratic lift}\label{app:lift}
We collect the standard value-function properties, prove
Proposition~\ref{prop:lift}, and verify the actual condition numbers of the
constructed instances.

\subsection{Standard value-function properties}
The following lemma supplies the facts used in Section~\ref{sec:setup} and
in the Moreau-envelope argument of Appendix~\ref{app:moreau}.
The sensitivity and differentiability statements are standard; see
\citet[Lemma~1]{sinha2018certifying} and
\citet[Lemma~4.3]{lin2020gradient}.
The envelope formula is given in \citet[Lemma~17]{lin2020near}.

\begin{lemma}[Primal function and Moreau envelope]
\label{lem:primal-envelope-properties}
Let $M,\mu>0$, and let
$f:\R^{d_x}\times\R^{d_y}\to\R$ be continuously differentiable,
with an $M$-Lipschitz gradient, and suppose $f(x,\cdot)$ is
$\mu$-strongly concave for every $x$.
Then $y^*(x):=\arg\max_y f(x,y)$ exists uniquely, the value function
$\Phi(x):=\max_y f(x,y)$ is continuously differentiable, and
\begin{equation}\label{eq:standard-primal-properties}
 \operatorname{Lip}(y^*)\le\frac M\mu,\qquad
 \nabla\Phi(x)=\nabla_x f(x,y^*(x)),\qquad
 \operatorname{Lip}(\nabla\Phi)\le M\left(1+\frac M\mu\right).
\end{equation}
Moreover, $\Phi$ is $M$-weakly convex.
If $\inf\Phi>-\infty$, then
\[
 u_x:=\arg\min_u\{\Phi(u)+M\norm{u-x}^2\}
\]
exists uniquely, and the Moreau envelope $P_f:=\Phi_{1/(2M)}$ satisfies
\begin{equation}\label{eq:standard-envelope-gradient}
 \nabla P_f(x)=2M(x-u_x)=\nabla\Phi(u_x).
\end{equation}
\end{lemma}
\begin{proof}
Strong concavity gives
\[
 f(x,y)\le f(x,0)+\ip{\nabla_y f(x,0)}{y}-\frac\mu2\norm y^2
 \longrightarrow-\infty\quad(\norm y\to\infty),
\]
so the inner maximum is attained uniquely on the full space.
The joint smoothness assumption bounds each blockwise Lipschitz constant by
$M$. Applying \citet[Lemma~1]{sinha2018certifying} with these constants gives
\eqref{eq:standard-primal-properties}.
For each fixed $y$, the function $f(\cdot,y)+M\norm{\cdot}^2/2$ is convex;
its pointwise supremum is $\Phi+M\norm{\cdot}^2/2$, proving weak convexity.
When $\inf\Phi>-\infty$, the proximal objective is coercive and
$M$-strongly convex, so $u_x$ exists uniquely.
\citet[Lemma~17]{lin2020near} gives the first equality in
\eqref{eq:standard-envelope-gradient}; the second is the first-order
optimality condition at $u_x$.
\end{proof}

\subsection{Proof of Proposition~\ref{prop:lift}}
The smoothness, gap, and oracle estimates are used in
Appendices~\ref{app:bv-calibration} and~\ref{app:as-calibration};
the stationarity transfer is used in
Proposition~\ref{prop:canonical-hardness}.

\begin{proof}[Proof of Proposition~\ref{prop:lift}]
\textit{Strong concavity, smoothness, and gap.}
Since $\alpha\beta=q$ and $\alpha\beta^2=h$,
\begin{equation}\label{eq:lift-derivatives}
 \nabla_x f_H=\nu\gamma(y-\gamma x),\qquad
 \nabla_y f_H=q\nabla H(\beta y)-\nu(y-\gamma x).
\end{equation}
The bound $\norm{\nabla^2H}\le\ell$ gives
$\nabla_{yy}^2f_H=h\nabla^2H(\beta y)-\nu I\preceq-\mu I$.
The Hessian of the quadratic coupling is
\[
 -\nu\begin{pmatrix}\gamma^2I&-\gamma I\\-\gamma I&I\end{pmatrix},
\]
whose operator norm is $\nu(1+\gamma^2)$.
The remaining Hessian has norm at most $\ell h$, proving
\eqref{eq:lift-smooth}.
Lemma~\ref{lem:primal-envelope-properties} now gives a unique inner
maximizer and the primal-gradient formula.

At the origin, smoothness of $H$ yields
\[
 f_H(0,y)\le\alpha H(0)+q\ip{\nabla H(0)}{y}-\frac\mu2\norm y^2,
 \qquad
 \Phi_H(0)\le\alpha H(0)+\frac{q^2g^2}{2\mu}.
\]
For any $x$, the feasible choice $y=\gamma x$ gives
\begin{equation}\label{eq:finite-lower-bound}
 \Phi_H(x)\ge\alpha H(\beta\gamma x)\ge\alpha\inf H>-\infty.
\end{equation}
Subtracting this lower bound from the bound at the origin proves
\eqref{eq:lift-gap}.

\medskip\noindent
\textit{Explicit stationarity transfer.}
Set $z^*=\beta y^*(x)$ and $w=\beta\gamma x$.
The inner first-order condition and the primal-gradient formula give
\begin{equation}\label{eq:transfer-identities}
 z^*-w=\frac h\nu\nabla H(z^*),\qquad
 \nabla\Phi_H(x)=\gamma q\nabla H(z^*).
\end{equation}
Consequently,
\[
 \norm{\nabla H(w)}
 \le\norm{\nabla H(z^*)}+\ell\norm{w-z^*}
 \le\left(1+\frac{\ell h}{\nu}\right)\norm{\nabla H(z^*)},
\]
which proves~\eqref{eq:stationarity-transfer}.
For the scaled function $H_{\rm sc}(u):=\alpha H(\beta u)$, this also gives
\[
 \norm{\nabla H_{\rm sc}(\gamma x)}
 \le\frac{1+\ell h/\nu}{\gamma}\norm{\nabla\Phi_H(x)}.
\]

\medskip\noindent
\textit{Stochastic oracle.}
The error identity~\eqref{eq:noise-identity} and unbiasedness of $g_H$ give
unbiasedness of $G$ and the variance bound $q^2v^2$.
If $g_H$ has root mean-squared Lipschitz constant at most $S$, then
Minkowski's inequality gives
\begin{align*}
 \bigl(\E\norm{G(z;\xi)-G(z';\xi)}^2\bigr)^{1/2}
 &\le\nu(1+\gamma^2)\norm{z-z'}+qS\norm{\beta(y-y')}\\
 &\le[\nu(1+\gamma^2)+hS]\norm{z-z'},
\end{align*}
which concludes the proof.
\end{proof}

\subsection{Actual condition numbers}
The next lemma is used in the class-membership checks in
Appendices~\ref{app:bv-calibration} and~\ref{app:as-calibration}.
It establishes the actual-condition-number statements in
Theorems~\ref{thm:bv} and~\ref{thm:as}.

\begin{lemma}[Condition numbers of the lift]\label{lem:actual-condition}
Consider the construction of Proposition~\ref{prop:lift}.
Suppose $\gamma^2=M/(4\mu)$, $\ell h\le\mu/4$, and
$\operatorname{Lip}(\nabla f_H)\le M$.
Let $\mu_{\rm act}$ be the largest uniform inner strong-concavity modulus,
and let $L_{\rm act}:=\operatorname{Lip}(\nabla f_H)$.
Then, for $\kappa=M/\mu$,
\begin{equation}\label{eq:actual-ratios}
 \frac\kappa6\le\frac{L_{\rm act}}{\mu_{\rm act}}\le\kappa.
\end{equation}
If the lifted oracle also has root mean-squared Lipschitz constant at most
$M$, its smallest such constant $\Lb_{\rm act}$ satisfies the same bounds:
$\kappa/6\le\Lb_{\rm act}/\mu_{\rm act}\le\kappa$.
\end{lemma}
\begin{proof}
The inner Hessian bounds imply
\begin{equation}\label{eq:actual-mu}
 \mu\le\mu_{\rm act}\le\nu+\ell h\le\frac{3\mu}{2}.
\end{equation}
Holding $y$ fixed and changing $x$ gives, for every seed,
\begin{equation}\label{eq:actual-lower}
\begin{aligned}
 \norm{\nabla f_H(x,y)-\nabla f_H(x',y)}
 &=\norm{G(x,y;\xi)-G(x',y;\xi)}\\
 &=\nu\gamma\sqrt{1+\gamma^2}\norm{x-x'}.
\end{aligned}
\end{equation}
Since $\nu\gamma\sqrt{1+\gamma^2}\ge\nu\gamma^2\ge M/4$,
we obtain $L_{\rm act}\ge M/4$ and, when finite, $\Lb_{\rm act}\ge M/4$.
Combining these estimates with~\eqref{eq:actual-mu} and the assumed upper
bounds proves the claims.
\end{proof}

\section{The base chain and its stochastic oracles}\label{app:base}
\subsection{The nonconvex chain}
Recall that we define $H=F_T$ in
\eqref{eq:base-chain-main} where $F_T$ is adapted from \citet{carmon2020}:
for $T\ge2$, define
\[
 \Psi(t)=\begin{cases}
  0,&t\le1/2,\\
  \exp\bigl(1-(2t-1)^{-2}\bigr),&t>1/2,
 \end{cases}
 \qquad
 \Phi_0(t)=\sqrt e\int_{-\infty}^t e^{-s^2/2}\,ds,
\]
and
\begin{equation}\label{eq:base-chain}
 F_T(u)=-\Psi(1)\Phi_0(u_1)
 +\sum_{i=2}^T\bigl[\Psi(-u_{i-1})\Phi_0(-u_i)
                   -\Psi(u_{i-1})\Phi_0(u_i)\bigr].
\end{equation}
For $a\ge0$, write
$\operatorname{prog}_a(u)=\max(\{i:|u_i|>a\}\cup\{0\})$.
Both $\Psi$ and $\Phi_0$ are bounded, so $F_T$ is bounded below.
We use the following published properties.

\begin{lemma}[Properties of $F_T$~\citep{arjevani2023}]\label{lem:base}
The function $F_T$ is twice continuously differentiable and satisfies
\begin{align}
 F_T(0)-\inf F_T&\le12T,\label{eq:bounded_init}\\
 \operatorname{Lip}(\nabla F_T)&\le152,\quad
 \norm{\nabla F_T(u)}_\infty\le23,\label{eq:base-bounds}\\
 \operatorname{prog}_0(\nabla F_T(u))
 &\le\operatorname{prog}_{1/2}(u)+1,\label{eq:base-support}\\
 \operatorname{prog}_1(u)<T
 &\ \Longrightarrow\ \norm{\nabla F_T(u)}>1.\label{eq:base-obstruction}
\end{align}
In particular, $\norm{\nabla F_T(0)}\le23$.
\end{lemma}
Based on the chain of \citet{carmon2020}.
The bound at zero follows from~\eqref{eq:base-support} and the coordinate bound.
We take this lemma as the external analytic input; the lift, oracle estimates,
progress argument, and parameter choices are proved here.

\subsection{A bounded-variance oracle}
Let $Z\sim\operatorname{Bernoulli}(p)$ with $0<p\le1$ and define
\begin{equation}\label{eq:bv-base-oracle}
 [g_T^{\rm BV}(u;Z)]_i
 =\nabla_iF_T(u)
 \left[1+\mathbf1\{i>\operatorname{prog}_{1/4}(u)\}(Z/p-1)\right].
\end{equation}
Since $\E(Z/p-1)=0$, the estimator is unbiased.
Moreover, $\operatorname{prog}_{1/2}(u)\le\operatorname{prog}_{1/4}(u)$.
By~\eqref{eq:base-support}, its error can have at most one nonzero coordinate,
namely $\operatorname{prog}_{1/4}(u)+1$ when this index is at most $T$.
Hence
\begin{equation}\label{eq:bv-base-variance}
 \E\norm{g_T^{\rm BV}(u;Z)-\nabla F_T(u)}^2
 \le23^2\frac{1-p}{p}.
\end{equation}
The new frontier coordinate can be nonzero only when $Z=1$.
This is the bounded-variance construction in
\citet[equation~(17)]{arjevani2023}; the argument above verifies the properties
needed for the lift.

\subsection{An averaged-smooth oracle}
We specify a single smooth gate before choosing any problem parameters.
Let
\[
 \Lambda(t)=\begin{cases}
 \displaystyle\exp\!\left(-\frac{1}{100(t-1/4)(1/2-t)}\right),
                   &1/4<t<1/2,\\[2pt]
 0,&\text{otherwise},
 \end{cases}
 \qquad
 \Gamma(t)=\frac{\int_{1/4}^{t}\Lambda(s)\,ds}
                 {\int_{1/4}^{1/2}\Lambda(s)\,ds}.
\]
Then $0\le\Gamma\le1$, $\Gamma=0$ on $(-\infty,1/4]$,
$\Gamma=1$ on $[1/2,\infty)$, and
$m_\Gamma:=\sup_t|\Gamma'(t)|<\infty$ is a universal constant.
This is the smooth gate used by \citet{arjevani2023}.
For the tail $u_{\ge i}=(u_i,\ldots,u_T)$, set
\begin{equation}\label{eq:as-base-oracle}
 \Theta_i(u)=\Gamma\bigl(1-\norm{\Gamma(|u_{\ge i}|)}\bigr),\qquad
 [g_T^{\rm AS}(u;Z)]_i
 =\nabla_iF_T(u)[1+\Theta_i(u)(Z/p-1)],
\end{equation}
where the inner $\Gamma$ is applied coordinatewise.
The flat regions of $\Gamma$ give
\begin{equation}\label{eq:gate-sandwich}
 \mathbf1\{i>\operatorname{prog}_{1/4}(u)\}
 \le\Theta_i(u)\le
 \mathbf1\{i>\operatorname{prog}_{1/2}(u)\}.
\end{equation}
Indeed, a tail bounded by $1/4$ gives inner norm zero and outer value one;
a tail containing an entry larger than $1/2$ gives inner norm at least one
and outer value zero.

\begin{lemma}[AS oracle moments]\label{lem:as-moments}
For $g=23$ and any $\ell\ge152$, the oracle~\eqref{eq:as-base-oracle} is unbiased,
and for all $u,v$,
\begin{align}
 \E\norm{g_T^{\rm AS}(u;Z)-\nabla F_T(u)}^2
 &\le g^2\frac{1-p}{p},\label{eq:as-base-variance}\\
 \E\norm{g_T^{\rm AS}(u;Z)-g_T^{\rm AS}(v;Z)}^2
 &\le\frac{4g^2m_\Gamma^4+3\ell^2}{p}\norm{u-v}^2.\label{eq:as-base-smooth}
\end{align}
The constants are independent of $T$.
\end{lemma}
\begin{proof}
Write $a_i(u)=\Theta_i(u)\nabla_iF_T(u)$.
By~\eqref{eq:base-support} and~\eqref{eq:gate-sandwich}, $a(u)$ has at most one
nonzero coordinate, $\operatorname{prog}_{1/2}(u)+1$, with magnitude at most $g$.
The centered error is $(Z/p-1)a(u)$, which proves unbiasedness
and~\eqref{eq:as-base-variance}.
The map $\Theta_i$ is $m_\Gamma^2$-Lipschitz: absolute value and the norm are
$1$-Lipschitz, and the inner and outer gates each contribute $m_\Gamma$.
The union $J=\operatorname{supp}(a(u))\cup\operatorname{supp}(a(v))$ has at most
two indices. On this union,
\[
 |a_i(u)-a_i(v)|^2
 \le2g^2|\Theta_i(u)-\Theta_i(v)|^2
       +2|\nabla_iF_T(u)-\nabla_iF_T(v)|^2.
\]
Summing over $J$ yields
\[
 \norm{a(u)-a(v)}^2
 \le(4g^2m_\Gamma^4+2\ell^2)\norm{u-v}^2.
\]
The zero mean of $Z/p-1$ eliminates the cross term, so
\begin{align*}
 \E\norm{g_T^{\rm AS}(u;Z)-g_T^{\rm AS}(v;Z)}^2
 &=\norm{\nabla F_T(u)-\nabla F_T(v)}^2
       +\frac{1-p}{p}\norm{a(u)-a(v)}^2\\
 &\le\frac{4g^2m_\Gamma^4+3\ell^2}{p}\norm{u-v}^2.
\end{align*}
The deliberately conservative constant avoids relying on a sharper numerical
estimate for the gate.
\end{proof}

\section{Support bounds and the stationarity lower bound}\label{app:progress}
We prove Lemma~\ref{lem:progress} and the stationarity bound used in
Appendices~\ref{app:bv-calibration}-\ref{app:moreau}.
Write $[m]=\{1,\ldots,m\}$ for integers $m\ge0$, with $[0]=\varnothing$.

\begin{lemma}[PZC property of the base oracles
{\citep[Lemmas~3-4]{arjevani2023}}]\label{lem:oracle-pzc}
For every $T\ge2$ and $p\in(0,1]$, the oracles
$g_T^{\rm BV}$ in~\eqref{eq:bv-base-oracle} and
$g_T^{\rm AS}$ in~\eqref{eq:as-base-oracle}
are probability-$p$ zero-chains in the sense of Definition~\ref{def:pzc}.
For either oracle, with $Z\sim\operatorname{Bernoulli}(p)$,
\[
 \operatorname{prog}_0(g_T(u;Z))
 \le \operatorname{prog}_{1/4}(u)+\mathbf1\{Z=1\}
 \qquad\text{for every }u\in\R^T\text{ and }Z\in\{0,1\}.
\]
\end{lemma}

\begin{proof}[Proof of Lemma~\ref{lem:progress}]
Let $\mathcal E$ be the first event in Definition~\ref{def:pzc}.
Outside a null set of seeds, that definition gives, simultaneously for all $u$,
\[
 \operatorname{prog}_0(g_H(u;\xi))
 \le \operatorname{prog}_{1/4}(u)+\mathbf1\{\xi\in\mathcal E\}
 \le \operatorname{prog}_0(u)+\mathbf1\{\xi\in\mathcal E\}.
\]
Set $J_0=0$ and
$J_t=\sum_{s=1}^t\mathbf1\{\xi_s\in\mathcal E\}$.
Write $G_b^{t,k}=G_b(z^{t,k};\xi_t)$ for $b\in\{x,y\}$.
Induction gives, almost surely, for every actual query and response,
\begin{equation}\label{eq:progress-increment}
\begin{aligned}
 \operatorname{supp}(x^{t,k})\cup\operatorname{supp}(y^{t,k})
       &\subseteq[J_{t-1}],\\
 \operatorname{supp}(G_x^{t,k})\cup\operatorname{supp}(G_y^{t,k})
       &\subseteq[J_t].
\end{aligned}
\end{equation}
The first queries are at the origin; thereafter, zero-respecting gives the
first inclusion from the earlier response supports.
The coupling terms in~\eqref{eq:lift-oracle} are supported within $[J_{t-1}]$.
Since $\beta>0$, the PZC inequality gives
$\operatorname{supp}(g_H(\beta y^{t,k};\xi_t))\subseteq[J_t]$.
This proves the second inclusion for every $k$ with the same seed $\xi_t$.

If the algorithm stops after $\tau\le R$ rounds, the output restriction gives
$\operatorname{supp}(\hat x)\subseteq[J_\tau]\subseteq[J_R]$;
unused seeds through round $R$ may be sampled independently for this argument.
Since $\Prob(\xi_t\in\mathcal E)\le p$, we have $\E J_R\le pR$.
Markov's inequality now yields
\[
 \Prob\!\left(\operatorname{supp}(\hat x)\not\subseteq[T-1]\right)
 \le\Prob(J_R\ge T)\le\frac{pR}{T}. \qedhere
\]
\end{proof}

\begin{proposition}[Stationarity lower bound]\label{prop:canonical-hardness}
Let $H=F_T$ with $T\ge2$, let $\eps>0$, and use either base oracle from
Lemma~\ref{lem:oracle-pzc} in the lift~\eqref{eq:lift-oracle}.
Suppose $\gamma q=8\eps$ and $\ell h/\nu\le1/5$.
For every admissible zero-respecting algorithm using at most $R\le T/(4p)$ rounds,
\begin{equation}\label{eq:canonical-event}
 \Prob\!\left(\norm{\nabla\Phi_H(\hat x)}>\frac{20}{3}\eps\right)
 \ge\frac34.
\end{equation}
Consequently,
\begin{equation}\label{eq:canonical-moments}
 \E\norm{\nabla\Phi_H(\hat x)}\ge5\eps,
 \qquad
 \E\norm{\nabla\Phi_H(\hat x)}^2\ge\frac{100}{3}\eps^2.
\end{equation}
\end{proposition}
\begin{proof}
Lemmas~\ref{lem:oracle-pzc} and~\ref{lem:progress} give
$\Prob(\hat x_T=0)\ge3/4$.
On this event, $(\beta\gamma\hat x)_T=0$, so
Lemma~\ref{lem:base} gives $\norm{\nabla F_T(\beta\gamma\hat x)}>1$.
The stationarity transfer~\eqref{eq:stationarity-transfer} implies
\[
 \norm{\nabla\Phi_H(\hat x)}
 \ge \frac{\gamma q}{1+\ell h/\nu}
       \norm{\nabla F_T(\beta\gamma\hat x)}
 >\frac{20}{3}\eps.
\]
This proves~\eqref{eq:canonical-event}; taking first and second moments
gives~\eqref{eq:canonical-moments}.
\end{proof}

An algorithm using at most $N$ oracle calls has at most $N$ rounds.
Thus both accuracy criteria fail whenever $N\le T/(4p)$.
Appendices~\ref{app:bv-calibration} and~\ref{app:as-calibration} apply
Proposition~\ref{prop:canonical-hardness} after choosing $T$ and $p$;
Appendix~\ref{app:moreau} transfers the same event to the envelope gradient.

\section{Proof of Theorem~\ref{thm:bv}}\label{app:bv-calibration}
We keep a conservative constant choice so the parameter checks are shared with AS.
Fix universal constants
\begin{equation}\label{eq:universal-constants}
 \Delta_0=12,\quad \ell=155,\quad g=23,\quad
 s=\max\{336,\sqrt{4g^2m_\Gamma^4+3\ell^2}\},\quad C=256g^2.
\end{equation}
The choices $155$ and $336$ provide slack; no random-rotation estimates are used.
For $M=L$ and $\kappa=M/\mu$, choose
\begin{gather}
 \gamma^2=\kappa/4,\quad q=8\eps/\gamma,\quad q^2=256\eps^2/\kappa,
 \quad h=\mu/(4\ell),\quad \nu=5\mu/4,\label{eq:bv-exact}\\
 p=\begin{cases}1,&\sigma=0,\\
       \min\{1,q^2g^2/\sigma^2\},&\sigma>0,
    \end{cases}
 \qquad B_*:=\frac{\Delta h}{4\Delta_0q^2},\qquad T=\lfloor B_*\rfloor.
 \label{eq:common-calibration}
\end{gather}
Use $H=F_T$, $\alpha=q^2/h$, $\beta=h/q$, and $g_H=g_T^{\rm BV}$.

\paragraph{Class membership.}
The strong concavity follows from Proposition~\ref{prop:lift}.
Since $\ell h/\nu=1/5$ and $\kappa\ge8$,
\[
 \nu(1+\gamma^2)+\ell h
 =\frac{3\mu}{2}+\frac{5M}{16}\le\frac M2\le M.
\]
The variance is at most $q^2g^2(1-p)/p\le\sigma^2$.
For $p<1$, this follows from $q^2g^2/p=\sigma^2$;
for $p=1$, the stochastic error vanishes.
The actual condition number satisfies~\eqref{eq:actual-ratios}.

\paragraph{Gap and chain length.}
The two quantities to control are
\begin{equation}\label{eq:bv-gap-length}
 \frac{q^2g^2}{2\mu}=\frac{128g^2\eps^2}{M},\qquad
 B_* =\frac{M\Delta}{4096\Delta_0\ell\eps^2}.
\end{equation}
Choose
\begin{equation}\label{eq:bv-c0}
 c_0\le\min\left\{\frac{1}{256g^2},
                  \frac{1}{65536\Delta_0\ell}\right\}.
\end{equation}
Then $\eps^2\le c_0M\Delta$ makes the first term in~\eqref{eq:bv-gap-length}
at most $\Delta/2$ and gives $B_*\ge16$.
Thus $T\ge B_*/2\ge8$.
Also $\alpha\Delta_0T\le\Delta/4$, so~\eqref{eq:lift-gap} gives
$\Phi_H(0)-\inf\Phi_H\le3\Delta/4$.
The lower bound~\eqref{eq:finite-lower-bound} verifies that the infimum is finite.

\paragraph{Query lower bound.}
The calibrated parameters satisfy Proposition~\ref{prop:canonical-hardness}.
Moreover,
\begin{align*}
 \frac{T}{4p}
 &\ge\frac{M\Delta}{32768\Delta_0\ell\eps^2}
          \max\left\{1,\frac{\kappa\sigma^2}{C\eps^2}\right\}\\
 &\ge\frac{1}{32768\Delta_0\ell C}\,
       \frac{M\Delta}{\eps^2}
          \max\left\{1,\frac{\kappa\sigma^2}{\eps^2}\right\}.
\end{align*}
Every integer budget not exceeding $T/(4p)$ fails with the uniform margins
in~\eqref{eq:canonical-moments}. This proves Theorem~\ref{thm:bv} for both
moment criteria, with $c=(32768\Delta_0\ell C)^{-1}$ or any smaller constant.
The hard dimension is $d_x=d_y=T$.

\section{Proof of Theorem~\ref{thm:as}}\label{app:as-calibration}
Use the constants~\eqref{eq:universal-constants}, let $M=\Lb$ and
$\kappa=M/\mu$, and keep the definitions of $\gamma,q,p,B_*,T,\alpha,\beta$
in~\eqref{eq:bv-exact}-\eqref{eq:common-calibration}, except that now
\begin{equation}\label{eq:as-exact-h}
 h=\min\left\{\frac\mu{4\ell},\frac{M\sqrt p}{4s}\right\},
 \qquad\nu=\mu+\ell h.
\end{equation}
Use $H=F_T$ and $g_H=g_T^{\rm AS}$.
In particular, $B_*$ is computed using the new $h$.

\paragraph{Class membership.}
We have $\ell h\le\mu/4$, $\nu\le5\mu/4$, and $\ell h/\nu\le1/5$.
Strong concavity and the BV calculation are unchanged.
Lemma~\ref{lem:as-moments} and Proposition~\ref{prop:lift} give the joint
averaged-smoothness upper bound
\begin{equation}\label{eq:as-smooth-check}
 \nu(1+\gamma^2)+\frac{hs}{\sqrt p}
 \le\frac{5\mu}{4}\left(1+\frac\kappa4\right)+\frac M4
 \le\frac{23M}{32}<M.
\end{equation}
Population smoothness follows by Jensen's inequality.
\eqref{eq:actual-ratios} verifies both actual condition numbers.

\paragraph{Gap and integer chain length.}
As before, $q^2g^2/(2\mu)=128g^2\eps^2/M$, while
\begin{equation}\label{eq:as-chain-scale}
 B_* =\frac{M\Delta}{4096\Delta_0\eps^2}
       \min\left\{\frac1\ell,\frac{\kappa\sqrt p}{s}\right\}.
\end{equation}
A sufficient choice of universal accuracy constants is
\begin{equation}\label{eq:as-constants}
 c_0\le\min\left\{\frac{1}{256g^2},
             \frac{1}{65536\Delta_0\max\{\ell,s\}}\right\},
 \qquad c_1\le\frac{g}{4096\Delta_0s}.
\end{equation}
The first condition in~\eqref{eq:as-accuracy} bounds the gap correction by
$\Delta/2$ and makes the first term in~\eqref{eq:as-chain-scale}, including
its prefactor, at least $16$.
If $p=1$, it also makes the second term at least $16$, since $\kappa\ge8$.
If $p<1$, then
\[
 \sqrt p=\frac{qg}{\sigma}=\frac{16g\eps}{\sqrt\kappa\sigma},\qquad
 \frac{M\Delta}{4096\Delta_0\eps^2}\frac{\kappa\sqrt p}{s}
 =\frac{gM\Delta\sqrt\kappa}{256\Delta_0s\eps\sigma}\ge16
\]
by the second condition in~\eqref{eq:as-accuracy} and~\eqref{eq:as-constants}.
Thus $B_*\ge16$ in every case, $T\ge B_*/2\ge8$, and the initial primal gap
is at most $3\Delta/4$ exactly as in Appendix~\ref{app:bv-calibration}.

\paragraph{Rate calculation.}
Set $p_*=1$ when $\sigma=0$ and
$p_* =\min\{1,\eps^2/(\kappa\sigma^2)\}$ otherwise.
The calibrated probability obeys
\begin{equation}\label{eq:p-comparison}
 p_*\le p\le Cp_*,\qquad C=256g^2.
\end{equation}
Therefore
\begin{align}
 \frac{T}{4p}
 &\ge\frac{M\Delta}{32768\Delta_0\max\{\ell,s\}\eps^2}
       \min\{p^{-1},\kappa p^{-1/2}\}\notag\\
 &\ge\frac{M\Delta}{32768\Delta_0\max\{\ell,s\}C\eps^2}
       \min\{p_*^{-1},\kappa p_*^{-1/2}\}.\label{eq:as-final-budget}
\end{align}
The second inequality uses $C\ge1$; both terms are bounded below by $C^{-1}$
times their $p_*$ counterparts.
For $a=\kappa\sigma^2/\eps^2$, elementary case splitting gives
\[
 \min\{p_*^{-1},\kappa p_*^{-1/2}\}
 =\max\{1,\min\{a,\kappa\sqrt a\}\},
 \qquad \kappa\sqrt a=\kappa^{3/2}\sigma/\eps.
\]
Substitution in~\eqref{eq:as-final-budget} and
Proposition~\ref{prop:canonical-hardness} prove Theorem~\ref{thm:as}, with
$c=(32768\Delta_0\max\{\ell,s\}C)^{-1}$ or smaller.
The cases $a\le1$, $1<a<\kappa^2$, and $a\ge\kappa^2$ give~\eqref{eq:as-regimes}.
Neither the construction nor the progress bound depends on the batch-size bound
$K$. Thus $d_x=d_y=T$ and all constants are independent of $K$.

\section{Proof of the Moreau-envelope lower bounds}\label{app:moreau}

We prove Corollary~\ref{cor:envelope-main} in three steps: bound the
smoothness of the primal function on the hard family, compare its
gradient with the Moreau-envelope gradient, and apply
Proposition~\ref{prop:canonical-hardness}.
Throughout, $M=L$ for BV and $M=\Lb$ for AS.

\begin{proposition}[Primal smoothness]\label{prop:primal-curvature}
For the lift in Proposition~\ref{prop:lift},
\begin{equation}\label{eq:hard-primal-smoothness}
 \operatorname{Lip}(\nabla\Phi_H)
 \le \frac{\nu\gamma^2\ell h}{\mu}.
\end{equation}
Under either parameter choice in
Appendices~\ref{app:bv-calibration}-\ref{app:as-calibration},
this upper bound is at most $5M/64$.
\end{proposition}
\begin{proof}
The inner optimality condition is
\[
 q\nabla H(\beta y^*(x))=\nu\bigl(y^*(x)-\gamma x\bigr).
\]
For two points $x,x'$, write $y=y^*(x)$ and $y'=y^*(x')$.
Subtracting their optimality conditions gives
\[
 \nu(y-y')=\nu\gamma(x-x')
       +q\bigl(\nabla H(\beta y)-\nabla H(\beta y')\bigr).
\]
Since $H$ has an $\ell$-Lipschitz gradient and $q\beta=h$,
\[
 \nu\norm{y-y'}
 \le \nu\gamma\norm{x-x'}+\ell h\norm{y-y'}.
\]
Using $\nu-\ell h=\mu$ yields
\begin{equation}\label{eq:lift-maximizer-sensitivity}
 \norm{y^*(x)-y^*(x')}
 \le \frac{\nu\gamma}{\mu}\norm{x-x'}.
\end{equation}
The primal-gradient formula in
Lemma~\ref{lem:primal-envelope-properties} and the inner optimality
condition imply
\[
 \nabla\Phi_H(x)
 =\nu\gamma\bigl(y^*(x)-\gamma x\bigr)
 =\gamma q\nabla H(\beta y^*(x)).
\]
Consequently,
\begin{align*}
 \norm{\nabla\Phi_H(x)-\nabla\Phi_H(x')}
 &\le \gamma\ell h\norm{y^*(x)-y^*(x')}\\
 &\le \frac{\nu\gamma^2\ell h}{\mu}\norm{x-x'}.
\end{align*}
For both parameter choices,
$\ell h\le\mu/4$, $\nu\le5\mu/4$, and $\gamma^2=M/(4\mu)$.
Substituting these inequalities gives
\[
 \frac{\nu\gamma^2\ell h}{\mu}
 \le \frac{1}{\mu}\frac{5\mu}{4}
                    \frac{M}{4\mu}\frac{\mu}{4}
 =\frac{5M}{64}.
\]
\end{proof}

\begin{lemma}[Equivalence of stationarity measures on the hard family]
\label{lem:envelope-comparison}
Under either calibration in
Appendices~\ref{app:bv-calibration} and~\ref{app:as-calibration},
let $P_{f_H}=(\Phi_H)_{1/(2M)}$.
Then, for every $x$,
\[
\frac{123}{128}\|\nabla P_{f_H}(x)\|
\le \|\nabla\Phi_H(x)\|
\le \frac{133}{128}\|\nabla P_{f_H}(x)\|.
\]
\end{lemma}

\begin{proof}
Let $u_x$ minimize $\Phi_H(u)+M\|u-x\|^2$.
The Moreau-gradient identities give
$\nabla P_{f_H}(x)=\nabla\Phi_H(u_x)=2M(x-u_x)$.
By Proposition~\ref{prop:primal-curvature},
\[
\|\nabla\Phi_H(x)-\nabla P_{f_H}(x)\|
\le \frac{5M}{64}\|x-u_x\|
= \frac{5}{128}\|\nabla P_{f_H}(x)\|.
\]
The two bounds follow from the triangle and reverse
triangle inequalities.
\end{proof}

\begin{proof}[Proof of Corollary~\ref{cor:envelope-main}]
Fix either the BV or AS setting and choose exactly the same hard
instance and parameters as in its lower-bound proof in
Appendix~\ref{app:bv-calibration} or~\ref{app:as-calibration}.
An algorithm using at most $N\le T/(4p)$ oracle calls also uses at most
$T/(4p)$ rounds. Proposition~\ref{prop:canonical-hardness} therefore gives
\[
 \E\norm{\nabla\Phi_H(\hat x)}\ge5\eps,
 \qquad
 \E\norm{\nabla\Phi_H(\hat x)}^2\ge\frac{100}{3}\eps^2.
\]
Lemma~\ref{lem:envelope-comparison} holds pointwise at the same random
output $\hat x$. It follows that
\begin{align*}
 \E\norm{\nabla P_{f_H}(\hat x)}
 &\ge \frac{128}{133}\E\norm{\nabla\Phi_H(\hat x)}
 \ge \frac{640}{133}\eps>\eps,\\
 \E\norm{\nabla P_{f_H}(\hat x)}^2
 &\ge \left(\frac{128}{133}\right)^2
          \E\norm{\nabla\Phi_H(\hat x)}^2\\
 &\ge \frac{100}{3}\left(\frac{128}{133}\right)^2\eps^2
 >\eps^2.
\end{align*}
Thus neither Moreau-envelope criterion can be met within the same
oracle budget. The bounds on $T/(4p)$ already proved in
Appendices~\ref{app:bv-calibration} and~\ref{app:as-calibration} give the
two claimed complexity lower bounds, with the same parameter ranges.
The point $u_{\hat x}$ is used only in the gradient comparison; it is
neither queried nor returned by the algorithm.
\end{proof}

\end{document}